\documentclass[draft]{jotart}

\usepackage{amsmath}
\usepackage{amssymb}
\usepackage{mathtools}

\theoremstyle{proclaim}

\newtheorem{theorem}{Theorem}[section]
\newtheorem{proposition}[theorem]{Proposition}
\newtheorem{lemma}[theorem]{Lemma}
\newtheorem{corollary}[theorem]{Corollary}

\theoremstyle{statement}

\newtheorem{remark}[theorem]{Remark}

\newcommand\R{\mathbb{R}}
\newcommand\C{\mathbb{C}}

\newcommand\eps{\varepsilon}

\begin{document}

\issueinfo{vv}{n}{yyyy} 

\pagespan{101}{103}

\title{Commutators close to the identity}

\author{Terence Tao}
\address{UCLA Department of Mathematics, Los Angeles, CA 90095-1555.}
\email{tao@math.ucla.edu}


\begin{abstract}  Let $D,X \in B(H)$ be bounded operators on an infinite dimensional Hilbert space $H$.  If the commutator $[D,X] = DX-XD$ lies within $\eps$ in operator norm of the identity operator $1_{B(H)}$, then it was observed by Popa that one has the lower bound $\| D \| \|X\| \geq \frac{1}{2} \log \frac{1}{\eps}$ on the product of the operator norms of $D,X$; this is a quantitative version of the Wintner-Wielandt theorem that $1_{B(H)}$ cannot be expressed as the commutator of bounded operators.  On the other hand, it follows easily from the work of Brown and Pearcy that one can construct examples in which $\|D\| \|X\| = O(\eps^{-2})$.  In this note, we improve the Brown-Pearcy construction to obtain examples of $D,X$ with $\| [D,X] - 1_{B(H)} \| \leq \eps$ and $\| D\| \|X\| = O( \log^{5} \frac{1}{\eps} )$.  
\end{abstract}

\maketitle

%%%%%%%%%%%%%%%%%%%%%%%%%%%%%%%%%%%%%%%%%%%%%%%%%

\section*{Introduction}

Let $H$ be a real or complex Hilbert space, and let $B(H)$ be the Banach algebra of bounded operators on this space, equipped with the operator norm $\| \| = \| \|_{\mathrm{op}}$.  Given any two operators $D,X$ in this space, we can form their commutator $[D,X] \coloneqq DX - XD$.  It is a classical result of Wintner \cite{wintner} and Wielandt \cite{wielandt} that this commutator $[D,X]$ cannot equal the identity operator $1_{B(H)}$ of $B(H)$; indeed this result holds with $B(H)$ replaced by any other Banach algebra.  The requirement that $D,X$ be bounded is of course crucial, as the well known unbounded example $Df \coloneqq \frac{d}{dx}f$, $X \coloneqq xf$ on $L^2(\R)$ shows.  It was observed by Popa \cite{popa} that one has the following more quantitative version of the Wintner-Wielandt theorem:

\begin{theorem}\label{popa-thm}  Let $D,X \in B(H)$ be such that
$$ \| [D,X] - 1_{B(H)} \| \leq \eps $$
for some $\eps>0$.  Then one has
$$ \|D\| \|X \| \geq \frac{1}{2} \log \frac{1}{\eps}.$$
\end{theorem}

\begin{proof}  For the convenience of the reader we reproduce the argument from \cite{popa} here.  By multiplying $D$ by a constant and dividing $X$ by the same constant, we may normalise $\|D\|=\frac{1}{2}$.    If we write $[D,X] = 1_{B(H)}+E$, then $\|E\| \leq \eps$, and a routine induction shows that one has the identity
$$ [D,X^n] = nX^{n-1} + X^{n-1} E + X^{n-2} E X + \dots + E X^{n-1}$$
for any $n \geq 1$.  In particular, by the triangle inequality one has
$$ n \|X^{n-1} \| \leq \| [D,X^n] \| + n \eps \|X\|^{n-1}.$$
Since $\|D\| = \frac{1}{2}$, we have from the triangle inequality that $\| [D,X^n] \| \leq \|X^n\|$.  Applying this bound, dividing by $n!$, and then summing in $n$, we conclude that
$$ \sum_{n=0}^\infty \frac{\|X^n\|}{n!} \leq \sum_{n=1}^\infty \frac{\|X^n\|}{n!} + \eps \sum_{n=0}^\infty \frac{\|X\|^n}{n!}.$$
We can cancel the absolutely convergent sum $\sum_{n=1}^\infty \frac{\|X^n\|}{n!}$ to conclude that
$$ 1 \leq \eps \exp( \| X \| )$$
giving the claim.
\end{proof}

We remark that the above argument is valid in any Banach algebra, but in this paper we will restrict attention to the specific algebras $B(H)$.

In \cite[Remark 2.9]{popa}, Popa raised the question as to whether the bound $\frac{1}{2} \log \frac{1}{\eps}$ could be significantly improved.  When $H$ is finite dimensional the question is vacuous, since $[D,X]$ then has trace zero and thus must have at least one eigenvalue outside the disk $\{ z: |z-1| < 1 \}$, so that $\|[D,X] - 1_{B(H)}\| \geq 1$ for all $D,X \in B(H)$.  However, in infinite dimensions, it follows from the work of Brown and Pearcy \cite{bp} that $[D,X]$ can be made arbitrarily close to $1$ in operator norm for $D,X \in B(H)$.  In fact we have the following result, essentially due to Popa \cite{popa}:

\begin{proposition}\label{bpp}  Let $H$ be infinite dimensional.  Then for any $0 < \eps \leq 1$, there exists $D,X \in B(H)$ with
$$ \| [D,X] - 1_{B(H)} \| \leq \eps $$
and
$$ \| D \| \| X \| = O( \eps^{-2} ).$$
\end{proposition}

Here and in the sequel we use the asymptotic notation $A = O(B)$ to denote an estimate of the form $|A| \leq CB$ for an absolute constant $C$.

\begin{proof}  As $H$ is isometric to the direct sum $H \oplus H$ of two copies of $H$ with itself, $B(H)$ is isometric to $B(H \oplus H)$, which is in turn isometric to the algebra $M_2(B(H))$ of $2 \times 2$ matrices with elements in $B(H)$, again equipped with the operator norm.  By the results of Brown and Pearcy \cite{bp}, the unipotent matrix
$$ \begin{pmatrix} 1_{B(H)} & 1_{B(H)} \\ 0 & 1_{B(H)} \end{pmatrix} $$
can be expressed as a commutator $[D_1,X_1]$ of two bounded operators $D_1,X_1 \in M_2(B(H))$; an inspection of the arguments in \cite{bp} (see also \cite{az}, \cite{popa}) reveals that we have the operator norm bounds $\|D_1\|, \|X_1\| = O(1)$.  If we then conjugate $D_1$ and $X_1$ by the diagonal matrix
$$ S_\eps \coloneqq \begin{pmatrix} \eps 1_{B(H)} & 0 \\ 0 & 1_{B(H)} \end{pmatrix}$$
we see that the matrices $D := S_\eps D_1 S_\eps^{-1}$, $X \coloneqq S_\eps X_1 S_\eps^{-1}$ have commutator
$$ [D, X] = \begin{pmatrix} 1_{B(H)} & \eps 1_{B(H)} \\ 0 & 1_{B(H)} \end{pmatrix}.$$
By inspection of matrix coefficients, we see that $\| D\|, \|X\| = O(\eps^{-1})$.
\end{proof}

In this note we extend this construction to obtain a bound closer to that in Theorem \ref{popa-thm}:

\begin{theorem}\label{main}  Let $H$ be infinite dimensional.  Then for any $0 < \eps \leq\frac{1}{2}$, there exists $D,X \in B(H)$ with
$$ \| [D,X] - 1_{B(H)} \| \leq \eps $$
such that
$$ \| D \| \| X \| = O\left( \log^{5} \frac{1}{\eps} \right).$$
\end{theorem}

The exponent of $5$ could certainly be improved, but we have not optimised it here, as we do not believe that these arguments will be able to attain the exponent $1$ in Theorem \ref{popa-thm}.  Nevertheless, the author believes that the bound in Theorem \ref{popa-thm} is essentially optimal up to constants in the limit $\eps \to 0$.

We now briefly discuss the methods of proof.  Instead of the $2 \times 2$ matrix algebra $M_2(B(H))$ used in the proof of Proposition \ref{bpp}, we use the $n \times n$ matrix algebra $M_n(B(H))$ for some large $n$ (the optimal choice turns out to be comparable to $\log \frac{1}{\eps}$).  The strategy is to try to locate almost upper-triangular matrices 
$$
D, X =
\begin{pmatrix}
* & * & * & \dots & * & * \\
* & * & * & \dots & * & * \\
0 & * & * & \dots & * & * \\
\vdots & \vdots & \vdots & \ddots & \vdots & \vdots \\
0 & 0 & 0 & \dots & * & * \\
0 & 0 & 0 & \dots & * & *
\end{pmatrix},
$$
where the asterisks denote coefficients in $B(H)$ to be determined, whose commutator $[D,X]$ is equal to the identity except in the top right-corner:
$$
[D, X] =
\begin{pmatrix}
1_{B(H)} & 0 & 0 & \dots & 0 & * \\
0 & 1_{B(H)} & 0 & \dots & 0 & 0 \\
0 & 0 & 1_{B(H)} & \dots & 0 & 0 \\
\vdots & \vdots & \vdots & \ddots & \vdots & \vdots \\
0 & 0 & 0 & \dots & 1_{B(H)} & 0 \\
0 & 0 & 0 & \dots & 0  & 1_{B(H)}
\end{pmatrix}.
$$
If one then conjugates $D,X$ by the diagonal matrix $S_\mu \coloneqq \mathrm{diag}( \mu^{n-1}, \mu^{n-2}, \dots, 1 )$ for some scalar $\mu>0$, then one will obtain matrices $D_\mu := S_\mu D S_\mu^{-1}$, $X_\mu := S_\mu X S_\mu^{-1}$ whose norm is bounded by $O(\mu^{-1})$ for any fixed $n$, such that $\|[D_\mu,X_\mu] - 1_{M_n(B(H))}\|$ is bounded by $O( \mu^{n-1} )$ for a fixed $n$.  The $n=2$ case of this argument recovers Proposition \ref{bpp}, and by optimising in $n$ and $\mu$ we will obtain Theorem \ref{main}.

It remains to select the matrices $D,X$.  After some experimentation, the author found it convenient to work with matrices $X$ of the form
$$
X \coloneqq 
\begin{pmatrix}
0 & 0 & 0 & \dots & 0 & b_1 \\
1_{B(H)} & 0 & 0 & \dots & 0 & b_2 \\
0 & 1_{B(H)} & 0 & \dots & 0 & b_3 \\
\vdots & \vdots & \vdots & \ddots & \vdots & \vdots \\
0 & 0 & 0 & \dots & 0 & b_{n-1} \\
0 & 0 & 0 & \dots & 1_{B(H)} & b_n
\end{pmatrix}
$$
for some coefficients $b_1,\dots,b_n \in B(H)$ to be determined.  We remark that this ansatz is in fact not particularly restrictive, since any almost upper triangular matrix whose entries on the lower diagonal are all invertible can be conjugated to be of this form.  This can be seen by first conjugating by a diagonal matrix to make all the lower diagonal entries equal to the identity, then conjugating by an upper diagonal matrix to place the diagonal entries into the right form, and continuing upwards until reaching the desired ansatz.

In order for $[D,X]$ to take the desired form, one can calculate that $D$ must equal the matrix
$$
D = \begin{pmatrix}
v & 1_{B(H)} & 0 & \dots & 0 & b_1 u \\
u & v & 2 \cdot 1_{B(H)} & \dots & 0 & b_2 u \\
0 & u & v & \dots & 0 & b_3 u \\
\vdots & \vdots & \vdots & \ddots & \vdots & \vdots \\
0 & 0 & 0 & \dots & v & (n-1) 1_{B(H)} + b_{n-1} u \\
0 & 0 & 0 & \dots & u & v + b_n u
\end{pmatrix}
$$
for some further coefficients $u,v \in B(H)$, which need to solve the system of equations
\begin{equation}\label{system-0}
 [v, b_i] + [u, b_{i-1}] + i b_{i+1} + b_i [u, b_n] = 0
\end{equation}
for $i=2,\dots,n-1$, and also
\begin{equation}\label{system-n-0}
 [v, b_n] + [u, b_{n-1}] + b_n [u, b_n] = n \cdot 1_{B(H)}.
\end{equation}
The task then reduces to locating operators $u,v,b_1,\dots,b_n \in B(H)$ solving this system, and obtaining good bounds for the norms of these operators.  By a routine perturbative analysis involving the contraction mapping theorem, as well as a simple renormalisation, matters then reduce to locating operators $u,v \in B(H)$ for which the operator $T: B(H)^n \to B(H)^{n-1}$ defined by
$$ T(b_i)_{i=1}^n \coloneqq ( [v,b_i] + [u,b_{i-1}] )_{i=2}^n$$
has a bounded right inverse.  This would be impossible in finite dimensions, as the operators $[v,b_i] + [u,b_{i-1}]$ would necessarily be of zero trace in that case; but it turns out that if one uses the infinite dimensionality of $H$ to write $H = H_1 \oplus H_2$ for two orthogonal subspaces $H_1,H_2$ isometric to $H$, and lets $u,v \in B(H)$ be isometries from $H$ to $H_1$, $H_2$ respectively, then one will be able to construct such a right inverse using Neumann series.

\begin{remark}  (This remark is due to Tobias Fritz, see {\tt terrytao.wordpress.com/2018/04/11}.)  For any constants $C,\eps>0$, one can ask the question of whether there exist operators $D,X \in B(H)$ such that $\|D\|, \|X\| \leq C$ and $\| [D,X] - 1_{B(H)} \| \leq \eps$.  One can rephrase these constraints as semidefinite constraints 
$$C^2 - DD^* \geq 0; \quad C^2 - XX^* \geq 0; \quad \eps^2 - ([D,X] - 1_{B(H)}) ([D,X] - 1_{B(H)})^* \geq 0$$
where the ordering $\geq$ is in the sense of positive semidefinite operators.  One can then use semidefinite programming techniques as in \cite{pna}, applied to a Gram matrix 
$$(\langle P(D,X) \phi, Q(D,X) \phi \rangle)_{P,Q \in {\mathcal P}}$$ 
for some test state $\phi$ and some collection ${\mathcal P}$ of noncommutative monomials $P = P(D,X)$, to test if such inequalities are satisfiable for any given $C, \eps$.  Unfortunately, if one restricts the set of monomials ${\mathcal P}$ to a computationally feasible set, it does not appear that one obtains any non-trivial pairs of satisfiable $(C,\eps)$ in this fashion; see the comments at {\tt terrytao.wordpress.com/2018/04/11} for further discussion.
\end{remark}

The author is supported by NSF grant DMS-1266164 and by a Simons Investigator Award.  We thank Sorin Popa for suggesting this problem, and Tobias Fritz and Will Sawin for helpful comments.  We also thank the anonymous referee for pointing out several corrections (in particular, improving the exponent in Theorem \ref{main} from $16$ to $5$).  Part of this research was performed while the author was visiting the Institute for Pure and Applied Mathematics (IPAM), which is supported by the National Science Foundation.

\section{Proof of theorem}

Let $H$ be an infinite dimensional vector space, let $0 < \eps < 1$, and let $n \geq 1$ be a natural number depending on $\eps$ to be chosen later.  Let $M_n( B( H) )$ denote the Banach algebra of $n \times n$ matrices with entries in $B(H)$, equipped with the operator norm; this is isometric to $B( H^{\oplus n})$ and thus to $B(H)$.  For any statement $S$, we let $1_S$ denote its indicator, thus $1_S=1$ when $S$ is true and $1_S=0$ when $S$ is false.

The first step is to reduce to the system \eqref{system-0}, \eqref{system-n-0} mentioned in the introduction.  Actually for technical reasons it is convenient to also introduce a perturbative parameter $\delta>0$.  We begin with the following commutator calculation:

\begin{lemma}[Commutator calculation]\label{comm} Let $u, v, b_1,\dots,b_n \in B(H)$, and let $\delta > 0$.  Let $X = (X_{ij})_{1 \leq i, j \leq n} \in M_n(B(H))$ denote the matrix with entries
$$ X_{ij} \coloneqq 1_{B(H)} 1_{i=j+1} + \delta b_i 1_{j=n} $$
for $1 \leq i,j \leq n$, thus
$$
X = 
\begin{pmatrix}
0 & 0 & 0 & \dots & 0 & \delta b_1 \\
1_{B(H)} & 0 & 0 & \dots & 0 & \delta b_2 \\
0 & 1_{B(H)} & 0 & \dots & 0 & \delta b_3 \\
\vdots & \vdots & \vdots & \ddots & \vdots & \vdots \\
0 & 0 & 0 & \dots & 0 & \delta b_{n-1} \\
0 & 0 & 0 & \dots & 1_{B(H)} & \delta b_n
\end{pmatrix}.
$$
Let $D = (D_{ij})_{1 \leq i, j \leq n} \in M_n(B(H))$ denote the matrix with entries
$$
D_{ij} \coloneqq \frac{1}{\delta} u 1_{i=j+1} + \frac{1}{\delta} v 1_{i=j} + i 1_{B(H)} 1_{j=i+1} + \delta b_i u 1_{j=n},$$
thus
$$
D = \begin{pmatrix}
\frac{1}{\delta} v & 1_{B(H)} & 0 & \dots & 0 & \delta b_1 u \\
\frac{1}{\delta} u & \frac{1}{\delta} v & 2 \cdot 1_{B(H)} & \dots & 0 & \delta b_2 u \\
0 & \frac{1}{\delta} u & \frac{1}{\delta} v & \dots & 0 & \delta b_3 u \\
\vdots & \vdots & \vdots & \ddots & \vdots & \vdots \\
0 & 0 & 0 & \dots & \frac{1}{\delta} v & (n-1) 1_{B(H)} + \delta b_{n-1} u \\
0 & 0 & 0 & \dots & \frac{1}{\delta} u & \frac{1}{\delta} v + \delta b_n u
\end{pmatrix}.
$$
Then the commutator $[D,X] = ([D,X]_{ij})_{1 \leq i,j \leq n}$ has entries
$$ [D,X]_{ij} = 1_{i=j} + 
([v, b_i] + [u, b_{i-1}] + i \delta b_{i+1} + \delta b_i [u, b_n] - n \cdot 1_{B(H)} 1_{i=n}) 1_{j=n}
$$
with the conventions that $b_0=b_{n+1}=0$, thus
$$
[D,X] = 1_{M_n(B(H))} + \begin{pmatrix}
0 & 0 & 0 & \dots & 0 & [v,b_1] + \delta b_2 + \delta b_1 [u,b_n]  \\
0 & 0 & 0 & \dots & 0 & [v,b_2] + [u,b_1] + 2 \delta b_3 + \delta b_2 [u,b_n]  \\
0 & 0 & 0 & \dots & 0 & [v,b_3] + [u,b_2] + 3 \delta b_4 + \delta b_3 [u,b_n]  \\
\vdots & \vdots & \vdots & \ddots & \vdots & \vdots \\
0 & 0 & 0 & \dots & 0 & [v,b_{n-1}] + [u,b_{n-2}] + (n-1)\delta b_n + \delta b_{n-1} [u,b_n]  \\
0 & 0 & 0 & \dots & 0 & [v,b_{n}] + [u,b_{n-1}] + \delta b_{n} [u,b_n] - n \cdot 1_{B(H)}
\end{pmatrix}.
$$
\end{lemma}

\begin{proof}  The identity in the lemma is unchanged if one replaces $\delta,u,v,b_1,\dots,b_n$ with $1,\frac{1}{\delta} u, \frac{1}{\delta} v, \delta b_1, \dots, \delta b_n$ respectively.  Hence we may normalise $\delta=1$.
If we let $\mathrm{diag}(v) \in M_n(B(H))$ be the operator with diagonal entries $v$, thus
$$
\mathrm{diag}(v) =
\begin{pmatrix}
v & 0 & 0 & \dots & 0 & 0 \\
0 & v & 0 & \dots & 0 & 0 \\
0 & 0 & v & \dots & 0 & 0 \\
\vdots & \vdots & \vdots & \ddots & \vdots & \vdots \\
0 & 0 & 0 & \dots & v & 0 \\
0 & 0 & 0 & \dots & 0 & v
\end{pmatrix}
$$
then we clearly have
$$ [\mathrm{diag}(v), X] = 
\begin{pmatrix}
0 & 0 & 0 &\dots & 0 & [v,b_1] \\
0 & 0 & 0 &\dots & 0 & [v,b_2] \\
0 & 0 & 0 &\dots & 0 & [v,b_3] \\
\vdots & \vdots & \vdots & \ddots & \vdots & \vdots \\
0 & 0 & 0 & \dots & 0 & [v,b_{n-1}]\\
0 & 0 & 0 & \dots & 0 & [v,b_n]
\end{pmatrix}
$$
In a similar fashion, we can compute
\begin{align*}
[X \mathrm{diag}(u), X] &= X [\mathrm{diag}(u), X] \\
&= X \begin{pmatrix}
0 & 0 & 0 & \dots & 0 & [u,b_1] \\
0 & 0 & 0 & \dots & 0 & [u,b_2] \\
0 & 0 & 0 & \dots & 0 & [u,b_3] \\
\vdots & \vdots & \vdots & \ddots & \vdots & \vdots \\
0 & 0 & 0 & \dots & 0 & [u,b_{n-1}]\\
0 & 0 & 0 & \dots & 0 & [u,b_n]
\end{pmatrix} \\
&=
\begin{pmatrix}
0 & 0 & 0 & \dots & 0 & b_1 [u,b_n] \\
0 & 0 & 0 & \dots & 0 & b_2 [u,b_n] + [u,b_1] \\
0 & 0 & 0 & \dots & 0 & b_3 [u,b_n] + [u,b_2] \\
\vdots & \vdots & \vdots & \ddots & \vdots & \vdots \\
0 & 0 & 0 & \dots & 0 & b_{n-1} [u,b_n] + [u,b_{n-2}]\\
0 & 0 & 0 & \dots & 0 & b_n [u,b_n] + [u,b_{n-1}]
\end{pmatrix}.
\end{align*}
Finally, if we introduce the upper diagonal matrix $N$ with entries $1_{B(H)},2 \cdot 1_{B(H)},\dots,(n-1) \cdot 1_{B(H)}$, thus
$$ 
N \coloneqq 
\begin{pmatrix}
0 & 1_{B(H)} & 0 & \dots & 0 & 0 \\
0 & 0 & 2 \cdot 1_{B(H)} & \dots & 0 & 0 \\
0 & 0 & 0 & \dots & 0 & 0 \\
\vdots & \vdots & \vdots & \ddots & \vdots & \vdots \\
0 & 0 & 0 & \dots & 0 & (n-1) \cdot 1_{B(H)} \\
0 & 0 & 0 & \dots & 0 & 0
\end{pmatrix},
$$
then we have
$$ [N,X] = 
\begin{pmatrix}
1_{B(H)} & 0 & 0 & \dots & 0 & b_2 \\
0 & 1_{B(H)} & 0 & \dots & 0 & 2b_3 \\
0 & 0 & 1_{B(H)} & \dots & 0 & 3b_4 \\
\vdots & \vdots & \vdots & \ddots & \vdots & \vdots \\
0 & 0 & 0 & \dots & 1_{B(H)} & (n-1)b_n \\
0 & 0 & 0 & \dots & 0 & (1-n) 1_{B(H)}
\end{pmatrix}.
$$
Since
$$ D := \mathrm{diag}(v) + X \mathrm{diag}(u) + N $$
the claim follows.
\end{proof}

\begin{corollary}\label{reduct}  Let $u, v, b_1,\dots,b_n \in B(H)$.  Assume that for some $\delta>0$ we have the equations
\begin{equation}\label{system}
 [v, b_i] + [u, b_{i-1}] + i \delta b_{i+1} + \delta b_i [u, b_n] = 0
\end{equation}
for $i=2,\dots,n-1$, and also
\begin{equation}\label{system-n}
 [v, b_n] + [u, b_{n-1}] + \delta b_n [u, b_n] = n \cdot 1_{B(H)}
\end{equation}
Then, for any $\mu >0$, there exist matrices $D_\mu, X_\mu \in M_n(B(H))$ such that
$$ \| D_\mu \| \leq \frac{1}{\mu^2 \delta} \| u\| + \frac{1}{\mu \delta} \|v\| + n-1 + \delta \sum_{i=1}^n  \mu^{n-i-1} \|b_i\| \| u \| $$
$$ \| X_\mu \| \leq 1 + \delta \sum_{i=1}^n \mu^{n-i+1} \|b_i\| $$
$$ \| [D_\mu,X_\mu] - 1_{M_n(B(H))} \| \leq \mu^{n-1} \| [v,b_1] + \delta b_2 + \delta b_1 [u,b_n] \|.$$
\end{corollary}

\begin{proof}  Let $D,X$ be the matrices from Lemma \ref{comm}, and set $D_\mu := \frac{1}{\mu} S_\mu D S_\mu^{-1}$, $X_\mu := \mu S_\mu X S_\mu^{-1}$.  Thus
$$
D_\mu = \begin{pmatrix}
\frac{1}{\mu \delta} v & 1_{B(H)} & 0 & \dots & 0 & \mu^{n-2} \delta b_1 u \\
\frac{1}{\mu^2 \delta} u & \frac{1}{\mu \delta} v & 2 \cdot 1_{B(H)} & \dots & 0 & \mu^{n-3} \delta b_2 u \\
0 & \frac{1}{\mu^2 \delta} u & \frac{1}{\mu \delta} v & \dots & 0 & \mu^{n-4} \delta b_3 u \\
\vdots & \vdots & \vdots & \ddots & \vdots & \vdots \\
0 & 0 & 0 & \dots & \frac{1}{\mu \delta} v & (n-1) 1_{B(H)} + \delta b_{n-1} u \\
0 & 0 & 0 & \dots & \frac{1}{\mu^2 \delta} u & \frac{1}{\mu \delta} v + \mu^{-1} \delta b_n u
\end{pmatrix}
$$
and
$$
X_\mu = 
\begin{pmatrix}
0 & 0 & 0 & \dots & 0 & \mu^n \delta b_1 \\
1_{B(H)} & 0 & 0 & \dots & 0 & \mu^{n-1} \delta b_2 \\
0 & 1_{B(H)} & 0 & \dots & 0 & \mu^{n-2} \delta b_3 \\
\vdots & \vdots & \vdots & \ddots & \vdots & \vdots \\
0 & 0 & 0 & \dots & 0 & \mu^2 \delta b_{n-1} \\
0 & 0 & 0 & \dots & 1_{B(H)} & \mu \delta b_n
\end{pmatrix}
$$
and (by \eqref{system})
$$
[D_\mu, X_\mu] - 1_{M_n(B_H)} =
\begin{pmatrix}
0 & 0 & 0 & \dots & 0 & \mu^{n-1}([v,b_1] + \delta b_2 + \delta b_1 [u,b_n] ) \\
0 & 0 & 0 & \dots & 0 & 0 \\
0 & 0 & 0 & \dots & 0 & 0 \\
\vdots & \vdots & \vdots & \ddots & \vdots & \vdots \\
0 & 0 & 0 & \dots & 0 & 0 \\
0 & 0 & 0 & \dots & 0 & 0
\end{pmatrix}.
$$
The claim now follows from the triangle inequality.  (One could use some orthogonality to sharpen the bounds slightly if desired, but we will not do so here.)
\end{proof}

As mentioned in the introduction, we can write (without loss of generality) $H = H_1 \oplus H_2$ for some orthogonal subspaces $H_1,H_2$ isometric to $H$, and let $u,v \in B(H)$ be isometries from $H$ to $H_1,H_2$; in particular, $\|u\|=\|v\|=1$.  Writing $u^*, v^*$ for the adjoints of $u,v$, we easily verify the identities
\begin{equation}\label{usu}
u^* u = v^* v = uu^* + vv^* = 1; \quad u^* v = v^* u = 0.
\end{equation}
In particular, the map $z \mapsto (u^* z, v^* z)$ is a Hilbert space isometry from $H$ to $H \oplus H$, with inverse map $(z_1,z_2) \mapsto u z_1 + v z_2$; as a consequence, the map
\begin{equation}\label{iso}
x \mapsto \begin{pmatrix} u^* x u & u^* x v \\ v^* x u & v^* x v \end{pmatrix}
\end{equation}
is a Banach algebra isometry from $B(H)$ to $M_2(B(H))$, with inverse map
\begin{equation}\label{iso-2}
\begin{pmatrix} x_1 & x_2 \\ x_3 & x_4 \end{pmatrix} \mapsto ux_1 u^* + u x_2 v^* + v x_3 u^* + v x_4 v^*.
\end{equation}

Let $T: B(H)^n \to B(H)^{n-1}$ denote the linear operator
$$ T(b_i)_{i=1}^n \coloneqq ([v,b_i] + [u,b_{i-1}])_{i=2}^n.$$
We now construct a right inverse of $T$:

\begin{proposition}[Right inverse]\label{str} There exists a linear right-inverse $R: B(H)^{n-1} \to B(H)^n$ of $T$ obeying the bound
$$ \sup_{1 \leq i \leq n} \| (Rb)_i \| \leq 8 \sqrt{2} n^2 \sup_{2 \leq i \leq n} \| b_i \|$$
for all $b = (b_i)_{i=2}^n \in B(H)^{n-1}$, where $(Rb)_1,\dots,(Rb)_n$ are the coefficients of $Rb \in B(H)^n$.
\end{proposition}

\begin{proof}  Let $L: B(H)^{n-1} \to B(H)^n$ denote the linear operator
$$ L(x_i)_{i=2}^n \coloneqq \left(- \frac{1}{2} x_i v^* -\frac{1}{2} x_{i+1} u^*\right)_{i=1}^n$$
with the convention that $x_1=x_{n+1}=0$.  Then by \eqref{usu}, the composition $TL: B(H)^{n-1} \to B(H)^{n-1}$ can be split as $TL = 1-E$, where $E: B(H)^{n-1} \to B(H)^{n-1}$ is the operator
$$ E(x_i)_{i=2}^n \coloneqq \left(\frac{1}{2} (v x_i v^* + v x_{i+1} u^* + u x_{i-1} v^* + u x_i u^*)\right)_{i=1}^n.$$
Let us place a slightly weighted norm $\| \|'$ on $B(H)^{n-1}$ by the formula
$$ \| (x_i)_{i=2}^n \|' \coloneqq \sup_{2 \leq i \leq n} \left(2 - \frac{i^2}{n^2}\right)^{-1/2} \| x_i \|,$$
the key point being that the weight $2 - \frac{i^2}{n^2}$ ranges between $1$ and $2$ and is slightly concave in $i$.
If
$$ \| (x_i)_{i=2}^n \|' \leq 1$$
then by hypothesis we have
\begin{equation}\label{xi}
 \| x_i \| \leq \left(2 - \frac{i^2}{n^2}\right)^{1/2}
\end{equation}
for all $1 \leq i \leq n+1$.  For each $2 \leq i \leq n$, we see from the isometries \eqref{iso}, \eqref{iso-2} that the operator 
$\frac{1}{2} (v x_i v^* + v x_{i+1} u^* + u x_{i-1} v^* + u x_i u^*) \in B(H)$ has the same operator norm as the matrix
$$ \frac{1}{2} \begin{pmatrix} x_i & x_{i+1} \\ x_{i-1}  & x_i \end{pmatrix} \in M_2(B(H)).$$
This norm is in turn bounded by the operator norm of the real $2 \times 2$ matrix
$$ \frac{1}{2} \begin{pmatrix} \|x_i \| & \| x_{i+1} \| \\ \| x_{i-1} \| & \| x_i \| \end{pmatrix};$$
bounding this norm by the Frobenius norm and using \eqref{xi}, we conclude that
\begin{align*}
\| (Ex)_i \|^2 &\leq \frac{1}{4} \left( (2 - \frac{i^2}{n^2}) + (2 - \frac{(i+1)^2}{n^2}) + (2 - \frac{(i-1)^2}{n^2}) + (2 - \frac{i^2}{n^2})\right) \\
&= 2 - \frac{i^2}{n^2} - \frac{1}{2n^2} \\
&\leq (1 - \frac{1}{4n^2}) (2 - \frac{i^2}{n^2}) \\
&\leq \left(1 - \frac{1}{8n^2}\right)^2 \left(2 - \frac{i^2}{n^2}\right) 
\end{align*}
and hence
$$ \| Ex \|' \leq \left(1 - \frac{1}{8n^2}\right) \|x\|'$$
for all $x \in B(H)^{n-1}$.  By Neumann series, the operator $1-E$ is then invertible with
$$ \| (1-E)^{-1} x \|' \leq 8n^2 \|x\|'.$$
If we then set $R := L(1-E)^{-1}$, and note that the weights $\left(2 - \frac{i^2}{n^2}\right)^{-1/2}$ vary between $1/\sqrt{2}$ and $1$, we obtain the claim.
\end{proof}

Let $u,v$ be as in the above proposition, and set
\begin{equation}\label{delta-n5}
 \delta \coloneqq \frac{1}{2000 n^5}.
\end{equation}
The system \eqref{system}, \eqref{system-n} can be written as
$$ Tb = a + \delta F(b) + \delta G(b,b)$$
where $b = (b_i)_{i=1}^n$, $a \in B(H)^{n-1}$ is the constant
$$ a \coloneqq ( 0, \dots, 0, n ),$$
$F: B(H)^n \to B(H)^{n-1}$ is the linear operator
\begin{equation}\label{fbn}
 F(b_i)_{i=1}^n \coloneqq ( -b_2, -2b_3, \dots, -(n-1)b_n, 0 )
\end{equation}
and $G: B(H)^n \times B(H)^n \to B(H)^{n-1}$ is the bilinear operator
\begin{equation}\label{gbn}
 G((b_i)_{i=1}^n, (b'_i)_{i=1}^n) \coloneqq (-b_2 [u,b'_n], \dots, -b_n [u,b'_n]).
\end{equation}
To solve this equation, we use the following abstract lemma:

\begin{lemma}  Let $X,Y$ be Banach spaces, let $T, F: X \to Y$ be bounded linear operators, let $G: X \times X \to Y$ be a bounded bilinear operator (thus $\|G(x,x')\|_Y \leq \|G\| \|x\|_X \|x'\|_X$ for some finite quantity $\|G\|$), and let $a \in Y$.  Suppose that $T$ has a 
bounded linear right inverse $R: Y \to X$.  Then, if $\delta>0$ obeys the inequality
\begin{equation}\label{cor}
\delta (2 \|F\| \|R\| + 4 \|G\| \|R\|^2 \|a\|_Y) < 1
\end{equation} 
there exists $b \in X$ with $\|b\|_X \leq 2 \| R \| \|a\|_Y$ that solves the equation
$$ Tb = a + \delta F(b) + \delta G(b,b).$$
\end{lemma}

\begin{proof}  We use the ansatz $b = Rc$, then it suffices to find fixed point in the ball
$$ B \coloneqq \{ c \in Y: \|c\| \leq 2 \|a\|_Y \}$$
of the map
$$ \Phi \coloneqq c \mapsto a + \delta F(Rc) + \delta G( Rc, Rc ).$$
By the contraction mapping theorem, it suffices to show that $\Phi$ maps $B$ to $B$ with the bound $\| \Phi(c) - \Phi(c') \|_Y \leq \alpha \| c-c' \|_Y$ for some $\alpha < 1$.

Let $c \in B$, then by the triangle inequality
\begin{align*}
 \| \Phi(c) \| &\leq \| a \|_Y + \delta \|F\| \|R \| \|c\|_Y + \delta \|G\| \| R\|^2 \|c\|_Y^2 \\
&\leq \|a\|_Y + \delta ( 2 \|F\| \|R\| \|a\|_Y + 4 \|G\| \|R\|^2 \|a\|_Y^2 ) \\
&\leq 2 \|a\|_Y 
\end{align*}
by the hypothesis on $\delta$.  Thus $\Phi$ maps $B$ to itself.  If $c' \in B$, then we have
$$ \Phi(c) - \Phi(c') = \delta F(R(c-c')) + \delta G( R(c-c'), Rc ) + \delta G( R(c'), R(c-c'))$$
and thus by the triangle inequality
\begin{align*}
 \| \Phi(c) - \Phi(c') \| &\leq \delta \|F\| \|R \| \|c-c'\|_Y + \delta \|G\| \| R\|^2 \|c-c'\|_Y \|c\|_Y + \delta \|G\| \| R\|^2 \|c'\|_Y \|c-c'\|_Y  \\
&\leq \delta ( \|F\| \|R\| + 4 \|G\| \|R\|^2 \|a\|_Y^2 ) \|c-c'\|_Y, 
\end{align*}
and the claim follows by the hypothesis on $\delta$.
\end{proof}

We apply this proposition with $X = B(H)^n$ and $Y = B(H)^{n-1}$, with norms
$$ \| (b_i)_{i=1}^n \|_X \coloneqq \sup_{1 \leq i \leq n} \|b_i\|$$
and
$$ \| (c_i)_{i=2}^n \|_Y \coloneqq \sup_{2 \leq i \leq n} \|c_i\|.$$
From Proposition \ref{str} we have
$$ \|R\| \leq 8 \sqrt{2} n^2.$$
We also clearly have
$$ \|a\|_Y \leq n; \quad \|F\| \leq n-1 $$
and from the triangle inequality (and the bound $\|u\| \leq 1$)
$$ \|G\| \leq 2.$$
The condition \eqref{cor} is then implied by
$$ \delta ( 16 \sqrt{2} (n-1) n^2 + 1024 n^5) < 1$$
which follows from the choice \eqref{delta-n5}.  This gives a solution $b = (b_i)_{i=1}^n$ to \eqref{system}, \eqref{system-n} such that
$$ \|b_i \| \leq 2 \|R\| \|a\|_Y \leq 16 \sqrt{2} n^3$$
for all $1 \leq i \leq n$.  If we apply Corollary \ref{reduct} with $\mu = \frac{1}{2}$, we conclude that
\begin{align*}
 \| D_\mu \| &= O( n^5 ) \\
 \|X_\mu\| &= O( 1 ) \\
 \| [D_\mu,X_\mu] - 1_{M_n(B(H))} \| &= O( n^3 2^{-n} ).
\end{align*}
If we set $n = \lfloor C \log \frac{1}{\eps} \rfloor$ for a sufficiently large absolute constant $C$, we will then have
$$ \|D_\mu\| \|X_\mu \| = O( \log^{8} \frac{1}{\eps} )$$
and
$$ \| [D_\mu,X_\mu] - 1_{M_n(B(H))} \| \leq \eps.$$
Since $M_n(B(H))$ is isometric to $B(H)$, we obtain the claim.

\section{Open questions}

Let $K(H)$ denote the ideal of $B(H)$ consisting of compact operators, and let $\C + K(H)$ denote the subspace of $B(H)$ consisting of operators of the form $\lambda \cdot 1_{B(H)} + T$ for some $T \in K(H)$.  In \cite[Theorem 2.1]{popa}, it was shown that if $A \in B(H)$ obeys the bounds
$$ \|A \| = O(1) $$
and
\begin{equation}\label{ax}
 \| A \| = O( \operatorname{dist}( A, \C + K(H) )^{2/3} )
\end{equation}
then there exist operators $D,X \in B(H)$ with $\|D\|, \|X\| = O(1)$ such that $A = [D,X]$.  (In fact, a more general statement was proven in that paper, in which one works with properly infinite $W^*$-algebras with arbitrary centre, as opposed to an operator (or bounded sequence of operators) in $B(H)$; see \cite{popa} for further details.)  The question was posed in \cite[Remark 2.9]{popa} as to whether the condition \eqref{ax} can be relaxed.  Theorem \ref{main} suggests that it may be possible to replace \eqref{ax} with a logarithmic bound
$$
 \| A \| = O\left( \log^{-C} \frac{1}{\operatorname{dist}( A, \C + K(H) )} \right)
$$
for some absolute constant $C$, or at least
$$
 \| A \| = O\left( \operatorname{dist}( A, \C + K(H) )^c \right)
$$
for an arbitrary constant $c>0$.  There are however a number of technical issues that would need to be resolved if one wanted to adapt the methods in this paper to this problem.  A model case of the problem would be if $A \in M_n(B(H)) \equiv B(H)$ took the form
$$
A = 
\begin{pmatrix}
\lambda \cdot 1_{B(H)} + T & 0 & 0 & \dots & 0 & \eps S \\
0 & \lambda \cdot 1_{B(H)} & 0 & \dots & 0 & 0 \\
0 & 0 & \lambda \cdot 1_{B(H)} & \dots & 0 & 0 \\
\vdots & \vdots & \vdots & \ddots & \vdots & \vdots \\
0 & 0 & 0 & \dots & \lambda \cdot 1_{B(H)} & 0 \\
0 & 0 & 0 & \dots & 0 & \lambda \cdot 1_{B(H)}
\end{pmatrix}
$$
for some non-trivial operator $S$ of bounded operator norm $O(1)$ (e.g., an isometry of infinite deficiency), some $0 < \eps \leq 1$, a complex number $\lambda$ of size $O(\eps^c)$ for some small $c>0$, and a finite rank operator $T$ that is also of operator norm $O(\eps^c)$, where one assumes that $n$ is sufficiently large depending on $c$, and $\eps$ sufficiently small depending on $c$ and $n$.  The methods in this paper work best when $|\lambda|$ is reasonably large (in particular, much larger than $\eps$) and the finite rank operator $T$ is absent; however, even then there is a non-trivial difficulty because the contraction mapping arguments used in this paper do not easily allow one to prescribe the value of top right coefficient $[v,b_1] + \delta b_2 + \delta b_1 [u,b_n]$ appearing in Lemma \ref{comm}, and in particular to set it equal to $S$.  Furthermore, even if this difficulty is resolved, it seems that some additional argument would be required to handle the finite rank perturbation $T$, as well as the case when $\lambda$ is small (e.g. if $\lambda=0$).  We were unable to resolve these issues satisfactorily.   See also \cite[Remark 2.8]{popa} for some closely related model problems.

\end{document}